\documentclass[11pt]{article}

\usepackage[sort, comma ]{natbib}
\usepackage{graphicx}
\usepackage{amsmath, amssymb, dsfont, amsthm}
\usepackage{booktabs, multirow}
\usepackage{colortbl, color, array}

\usepackage{float}

\usepackage{mathtools}
\usepackage{graphics}
\usepackage{epstopdf}
\usepackage{multirow}
\usepackage{epsfig}
\usepackage{mathrsfs}
\usepackage{float}
\usepackage{lineno}
\usepackage{natbib}

\graphicspath{{Figures/}}
\usepackage[]{graphicx}
\chardef\bslash=`\\ 
\newcommand{\ntt}{\normalfont\ttfamily}

\newcommand{\pkg}[1]{{\protect\ntt#1}}

\DeclareMathOperator{\diag}{diag}
\newcommand{\cmle}{\widehat{\theta}_n^c}
\newcommand{\cmlebeta}{\widehat{\beta}_n^c}
\newcommand{\g}{\cellcolor[gray]{0.95}}

\usepackage{color}
\definecolor{purple}{rgb}{0.55,0.2,0.90}

\newtheorem{thm}{Theorem}[section]

\let\proglang=\textsf

\begin{document}
\title{Multiple Comparisons using Composite Likelihood in Clustered Data}
\author{Mahdis Azadbakhsh, Xin Gao and Hanna Jankowski}
\date{York University \\[\baselineskip] \today}
\maketitle

\begin{abstract}
We study the problem of multiple hypothesis testing for correlated clustered data.  As the existing multiple comparison procedures based on maximum likelihood estimation could be computationally intensive,  we propose to construct multiple comparison procedures based on composite likelihood method. The new test statistics account for the correlation structure within the clusters and are computationally convenient to compute. Simulation studies show that  the composite likelihood based procedures maintain good control of the familywise type I error rate in the presence of intra-cluster correlation, whereas ignoring the correlation leads to erratic performance. Using data arising from a depression study, we show how our composite likelihood approach makes an otherwise intractable analysis possible.   
\end{abstract}
\maketitle                   

\section{Introduction}

{The prevalence of depression in seniors estimated by the World Health Organization varies between 10\% to 20\%  \citep{baru}.  Understanding the relationship between depression and other health factors can help prevent the disease and alleviate the symptoms. The health and retirement study (HRS) conducted by the University of Michigan is a longitudinal study which measured various aspects of health, retirement and aging, including the subject's depression status. In this study, seniors were measured every two years from 1994 to 2012.  The objective of our analysis is to estimate the effect of several health factors known to be associated with depression status and compare the effect sizes of different factors.  Multiple comparisons on the effect sizes will clarify the relative importance of different factors to the disease. For example, the factor of being sleepless and the factor of smoking both are shown to attribute to the occurrence rate of depression. One might question whether or not they are equally important or one factor is more important than the other for development of the disease. Therefore, to fully understand the effects of the health factors, we perform all pairwise comparisons on their effect sizes.  }

{The repeated binary measurements of depression status observed in this data set are correlated within individuals. These repeated measurements can be viewed as clustered data since they are recorded from the same experimental unit multiple times.  Clustered data examples arise in many other situations, including measurements coming from siblings or same pedigrees, or measurements taken in close proximity to each other in spatial data.  Ignoring existing correlations within clusters leads to invalid individual or multiple inferences. }

{When performing multiple comparisons in clustered data, one should therefore, take into account the correlation structure within the clusters.   However, full likelihood analyses on such data often encounter computational challenges. For a repeated binary measurement data, the distribution can be described by multivariate probit or quadratic exponential. Evaluating the full likelihood of a multivariate probit model involves multi-dimensional integration, which quickly becomes computationally prohibitive. For the quadratic exponential model, the normalizing constant has to be computed through summation of all possible configurations of the clustered data, and here again computational intensity increases with the cluster size.  We can avoid this computational burden by using a composite likelihood approach.}

{Composite likelihood methods are extensions of the likelihood method that project high-dimensional likelihood functions to low-dimensional ones \citep{lind, coxReid}. This dimension reduction is achieved by compounding valid marginal or conditional densities.  It has been shown that, under regularity conditions, the composite likelihood estimator has desirable properties, such as consistency and asymptotic normality \citep{lind, vari, variReid, coxReid}. This makes it an appealing alternative in inferential procedures.   Furthermore, composite likelihood is more computationally convenient than full likelihood at a cost of some loss of efficiency.  The magnitude of this loss depends on the dimension of the multivariate vector and its dependency structure. Composite likelihood methodology has been applied to numerous statistical problems  \citep{zhaoJoe,renaMole,geysMole}, however,  the potential of composite likelihood in multiple testing has yet to be explored. 
There is a great need to develop a procedure to integrate multiple hypothesis testing procedures with the composite likelihood methodology.  }

Multiple testing procedures have been developed to control the overall type I error rate when the number of tests is greater than one \citep{hoch, bret}.  The classical Bonferroni method is the simplest procedure to adjust the overall type I error rate, but it is well-known to be very conservative. The Dunn-Sid\'ak procedure \citep{sida} generalizes the Bonferroni procedure by using a slightly less conservative $p$-value threshold for each comparison. \citet{sche} established  a method for testing all possible linear comparisons among a set of normally distributed variables, which tends to be over-conservative for a finite family of multiple comparisons.  Several stage-wise procedures have also been proposed to improve the power.  \citet{sime} modified the Bonferroni procedure based on ordered p-values. \citet{holm} proposed a multi-stage procedure that adjusts the family-wise error rate in each step using the number of remaining null hypotheses. \citet{homm}  suggested a stagewise rejective multiple test based on the principle of closed test procedures.  All of these methods are less conservative and therefore more powerful than the Bonferroni method.  However, it is difficult to construct simultaneous confidence intervals based on stage-wise procedures.    As another alternative, \citet*{hothBretWest}  proposed to use quantiles of the multivariate normal and multivariate t-distribution to perform multiple comparisons  in parametric methods. This method takes into account the correlation structure of the test statistics and offers sharper control of the family-wise type I error rate. The approach has been employed in many parametric and nonparametric settings to provide both multiple inferences and simultaneous intervals \citep{koniHothBrun, hothBretWest, koniBosiBrun}. {Recently, there has been considerable interest devoted to the problems of  large-scale multiple testing applied on the analysis of high dimensional data (see, for example, \citet{meijGoem, benjHoch}). New decision-analytic based multiple testing procedures \citep{lisoBurm}  have also been proposed to design multiple testing procedure to minimize a predefined utility function.}

In this paper, we propose a new procedure to handle multiple testing scenarios in computationally intensive or intractable likelihood scenarios.  We do this by combining multiple testing methods with the dimension-reduction capabilities of inference based on composite likelihood.  We explore in detail different multivariate models for correlated clustered data including the multivariate normal, multivariate probit, and quadratic exponential models to illustrate our multiple comparisons approach.  Although the proposed composite likelihood methodology can be combined with many multiple testing procedures, including the recent development in large scale multiple testing or decision-analytic based procedures, in this paper, we focus on combining the Bonferroni, Scheff\'e, Dunn-Sid\'ak, Holm, and the multivariate normal quantile (MNQ) of \citet{hothBretWest} methods with both univariate and conditional composite likelihood formulations.  Among these methods, the multivariate normal quantile threshold appears to have the best control of the familywise type I error rate in most simulation settings. 
 
{The structure of this paper is as follows:  In Section~\ref{sec:methods}, we develop our composite likelihood based test statistics for multiple inferences and establish their asymptotic properties.  In Section~\ref{sec:three-models}, we provide details on how to apply the general approach on a variety of multivariate models. In Section~\ref{sec:simulation}, we conduct simulation studies to evaluate empirical performance of the proposed method.  Finally, we analyze the depression data set to demonstrate the practical utility of the method.  This is done in Section~\ref{sec:data}. We conclude the paper with a brief discussion of the results.} 

\section{Multiple Comparisons Procedures based on Composite Likelihood}\label{sec:methods}

Let $\{f(Y;\theta),\theta\in\Theta\},$ where $\theta=(\theta_1,\dots,\theta_p)^T,$ be a parametric statistical model with parameter space $\Theta\subset \mathbb{R}^p.$  Let $Y=(y_1^T,\cdots,y_n^T)$ denote the response variables, where $y_i=(y_{i1},\cdots,y_{im_i})^T$  is the vector of observations from cluster $i$, $i = 1,\cdots, n$ from a study population.  It is assumed that observations from different clusters are independent, whereas observations from the same cluster may be dependent.  Note that each cluster is thus of size $m_i,$ for an overall sample size of $\sum_{i=1}^n m_i.$   In this work, it is assumed that the cluster size, $m_i$, is uniformly bounded.  

Let  $$C=C_{c\times p}=(C^{(1)},C^{(2)},\cdots,C^{(c)})^T$$ denote the contrast matrix.  A family of $c$ linear combinations of the parameters can then be specified by $C\theta$.   Let $H_{0i}$ denote the hypothesis that $C^{(i)}\theta=0,$ for $i=1, \ldots, c.$  We focus here on jointly testing the family of hypotheses $H_{0i}, i=1, \ldots, c.$ In multiple testing, the family-wise type I error (FWER) rate is the probability of falsely rejecting at least one individual null hypothesis.  It is said that one has weak control of FWER if the FWER $\leq \alpha$  when all of the null hypotheses are true,
whereas one has  strong control of FWER if the FWER $\leq \alpha$ under any combinations of null hypotheses and alternative hypotheses.   

{Most existing multiple comparisons methods estimate the parameters based on the full likelihood.  However, for some multivariate distributions, maximizing the full likelihood function can be computationally challenging.   Composite likelihood is an alternative method that has attracted much attention in recent years \citep{besa, lind, vari, variReid}.   A composite likelihood is a compounded form of marginal or conditional likelihoods, which is computationally easier to maximize.  The general formulation for composite likelihood may be written as follows:  Let $A_1,\dots,A_K$ be a suitably chosen collection of index sets, with $A_k\subseteq \{(i,j),i=1, \ldots, n,  j=1, \ldots, m_i\}.$  For each $A_k,$ a weight $w_{A_k}$ is also chosen/specified.  The composite likelihood function is then defined as 
\begin{eqnarray*}
CL(\theta;Y) &=& \prod_{k=1}^K f(y_{A_k}; \theta)^{w_{A_k}},
\end{eqnarray*} 
where $f(y_{A_k}; \theta)$ is the density for the subset vector $y_{A_k}.$  For example, to obtain the so-called univariate composite likelihood $CL(\theta;Y) = \prod_{i=1}^n \prod_{j=1}^{m_j} f(y_{ij})$, one chooses $A_k=\{(i,j)\}$ as the single index pairs and weights $w_{A_k}\equiv 1.$  (Note that the univariate composite likelihood is equivalent to the full likelihood if the $y_{ij}$ are independent.)
The so-called conditional composite likelihood is formulated as $CL(\theta;Y) = \prod_{i,j} f(y_{ij}|y_{i(-j)})=\prod_{i,j} f(y_{i})/f(y_{i(-j)}),$   where $ y_{i(-j)}$ denotes the sub-vector $ y_i$ with its $j$th element removed. This composite likelihood uses index sets $\{(i,1),\dots,(i,m_i)\}$ with weight $w_{A_k}=1$ and index sets $\{(i,1),\dots,(i,j-1),(i,j+1)$ $,\dots,(i,m_i) \}$ with weight $w_{A_k}=-1$.   As this example shows, the index sets $A_k$ need not be disjoint.}

The maximum composite likelihood estimate (MCLE) is defined as 
\begin{eqnarray*}
\cmle&=&\text{argmax}_{\theta\in \Theta} CL(\theta;Y).
\end{eqnarray*}  \citet{xuReid} give precise conditions under which $\cmle$ is consistent for $\theta.$   Under appropriate assumptions,  $\sqrt{n}(\cmle-\theta)$ is asymptotically normally distributed with mean zero and limiting variance given by the inverse of the the Godambe information matrix \citep{lind, variVido}, where 
\begin{eqnarray}\label{line:godambe}
G^{-1}(\theta) &=& H^{-1}(\theta) J(\theta) H^{-1}(\theta),
\end{eqnarray}

with $H(\theta)= \lim_nE(-cl^{(2)}(\theta;Y))/n$ and $J(\theta)=  \lim_n\mbox{var}(cl^{(1)}(\theta;Y))/n.$   Here, $cl^{(1)}$ is the vector of first derivatives and $cl^{(2)}$ is the matrix of second order derivatives of $cl(\theta;Y)=\log CL(\theta;Y) $ with respect to $\theta$.  The $H(\theta)$ can be estimated as the negative Hessian matrix evaluated at the maximum composite likelihood estimator, whereas the matrix $J(\theta)$ can be estimated as the sample covariance matrix of the composite score vectors. Both estimators $\widehat{H}_n$ and $\widehat{J}_n$ are consistent \citep{variVido}.

Consider the hypothesis test on a family of linear combinations of the parameters: $\{H_0:C \theta=0\}.$ Denote by $\Gamma=G^{-1}(\theta)$ the inverse Godambe information matrix, and let $\widehat \Gamma_n$ denote a consistent estimator of $\Gamma$ with $
\widehat \Gamma_n = \widehat H^{-1}_n \widehat J_n \widehat H^{-1}_n.$  We propose the following test statistics for our hypothesis test
\begin{eqnarray}\label{def:htests}
T_{i,n} &=& \frac{{C^{(i)}}^T\cmle}{\sqrt{\left({C^{(i)}}^T{\widehat \Gamma_n}C^{(i)}\right)/n}}, \ \ i=1, \ldots, c. 
\end{eqnarray}
\begin{thm}\label{thm:limit}
Suppose that the following conditions hold
\begin{enumerate}
\item $\sqrt{n}(\cmle-\theta) \Rightarrow N(0, G^{-1}(\theta)),$
\item $H_0$ is true, and 
\item $\widehat \Gamma_n \stackrel{p}{\rightarrow} G^{-1}(\theta).$
\end{enumerate}
Then the limiting distribution of $T_n=(T_{1,n},\cdots,T_{c,n})^T$ is multivariate normal $N_c(0,V),$ where 
\begin{eqnarray}\label{def:V}
V&=& \diag(D)^{-1/2} D \diag(D)^{-1/2}, \ \ \ D \ =\ C G^{-1}(\theta) C^T.
\end{eqnarray}
Furthermore, since $V_{i,i}=1,$ the marginal asymptotic distribution of each individual $T_{i,n}$ is standard normal.  
\end{thm}

\begin{proof}[Proof of Theorem \ref{thm:limit}]
Asymptotic normality of $T_n$ is shown as in \citet{hothBretWest}. Moreover, as the diagonal elements of the matrix $V$ are equal to one, the individual test statistics $T_{i,n},~i=1,\ldots,c,$ are standard normal.  Therefore, the $V$ matrix  is the correlation matrix for $ C\cmle$.   
\end{proof}

In practice, we estimate $V$ by plugging $\widehat \Gamma_n$ as a consistent estimator of $G^{-1}(\theta)$ into \eqref{def:V}. This  results in a consistent estimator of $V$. It is worthy to point out that the test statistics we propose here are Wald-type of statistics which are not invariant to reparametrization. Under reparametrization, the new statistics follow the same type of limiting distributions, but the values of the statistics are not the same. This is a standard limitation associated with Wald-type statistics.

To apply various multiple testing procedures, we propose to apply the corresponding rejection criterions based on the composite likelihood test statistics $\{T_{i,n}\}$ derived above. In the numerical analysis, we examine the performance of composite likelihood test statistics with  four popular multiple testing methods:  Bonferroni,  Dunn-Sid\'ak \citep{sida}, Holm \citep{holm}, and the simultaneous multiple comparison test based on multivariate normal quantiles (MNQ) of \citet{hothBretWest}.    The first three methods are applied to the marginal distributions of $T_{i,n}$ based on the asymptotic theory in Theorem~\ref{thm:limit}, whereas  MNQ uses a cutoff based on the multivariate quantile based on the full variance matrix $V$ in \eqref{def:V}.     

\section{Three multivariate models}\label{sec:three-models}

To showcase our methodology, we consider three different multivariate distributions: The multivariate normal, multivariate probit, and quadratic exponential distributions.  A further, fourth, distribution (the skewed multivariate gamma) is considered within the supplementary material.   For the first two distributions, the composite likelihood is constructed as sum of univariate likelihoods, whereas for the third distribution, the composite likelihood is constructed as conditional likelihood.  All details for the gamma example are given in the supplementary files.  Naturally, our methodology is not limited to these distributions and can be applied to other distributions, given that the composite likelihood is available and  that the conditions of Theorem~\ref{thm:limit} hold.

In order to include covariates into our modelling scheme, let $X_{i}$ denote an $m_i\times p$ matrix containing the values of $p$ covariates for the $m_i$ individuals in the $i^{th}$ cluster and $\beta=(\beta_{1},\dotsc,\beta_{p}){^T}$ denote the vector of regression coefficients.   Let {$\vec{x}_{ij}$} denote the $j^{th}$ row of the matrix $X_i$ (this is the vector of covariates for individual $j$ in cluster $i$).

\subsection{ Multivariate Gaussian distribution}

{Let $\{(y_{i},X_{i}),\ i=1,\cdots n\},$ denote the response and covariates arising from a multivariate normal model, with 
$
y_{i} = X_{i} \beta + \epsilon_i, i=1, \ldots, n,
$
and $m_i=m.$ We assume that \textrm{$\epsilon_i \thicksim N_m(0, \Sigma)$} where $\Sigma=(\sigma_{ij}), i,j=1,\dots, m,$ is an arbitrary covariance matrix. The univariate composite likelihood is thus equal to
\begin{eqnarray*}\label{line:cl1_norm}
&cl\left(\beta\right) \ \ = \ \ \sum_{i=1}^n \sum_{j=1}^m (-\frac{1}{2}\log(2\pi\sigma_{jj})-\frac{1}{2\sigma_{jj}^2}(y_{ij}-\vec{x}_{ij}\beta)^2),&
\end{eqnarray*} where the $\sigma_{jj}$'s are nuisance parameters.  The Hessian matrix and variability matrix are, respectively,  $H(\beta) = n^{-1}\left(\sum_{i=1}^n X_i^T W X_i\right)$ and $
J(\beta) =n^{-1}\left(\sum_{i=1}^n X_i^{T}W \, \Sigma \, W X_i\right),$ with $W=\diag(\Sigma)^{-1}.$ To estimate the regression coefficients, we employ an iterative algorithm:   Given the current estimate for the nuisance parameters $\sigma_{jj}$'s, we maximize the composite likelihood to obtain an estimate of $\widehat{\beta}^c_n= (\sum_{i=1}^{n} X_{i}^{T}W X_{i})^{-1}\sum_{i=1}^{n} X^{T}_{i} W Y_{i}$, and given a current estimate for $\beta,$ we use the sample covariance matrix of residuals to estimate $\Sigma.$  Based on the estimates $\widehat \beta^c_n$ and $\widehat{\Sigma},$ we obtain estimates for $H(\beta)$ and $J(\beta)$ with $W$ being replaced by its estimate $\widehat {W}= \diag(\widehat{\Sigma})$. }

\subsection{Multivariate probit model}\label{sec:probitdef}

{Let $y_{i}^* =X_{i}\beta +\epsilon_{i}$ with $\epsilon_{i} \thicksim N_m(0, \Sigma)$ and $\Sigma=\sigma R,$ where $R$ is an $m\times m$ correlation matrix.   The variables $y_{i}^*$ are the latent response variables, and  their  dichotomized version of the latent variable with
$ y_{ij}=I(y_{ij}^*>0),\ \ j=1,\cdots,m$
yield the multivariate probit model.  We therefore have that 
$
P(y_{ij}=1|X_i) \ \ = \ \ \Phi(\vec{x}_{ij}\beta/\sigma)$
where $\Phi$ denotes the univariate  standard normal cumulative distribution function. 
It follows that the parameters $\beta$ and $\sigma$ are not fully identifiable in the model, and we can only estimate the ratio $\beta/\sigma$.  To simplify notation, $\sigma$ is set equal to 1 in what follows.   The univariate composite log-likelihood function of the probit model is then formulated as
\begin{eqnarray*}
&cl(\beta;Y) \ \ = \ \ \sum_{i=1}^{n}\sum_{j=1}^m [ y_{ij} \ \log \ \Phi\left(\vec{x}_{ij} \beta\right)+(1-y_{ij}) \ \log \left(1- \Phi\left(\vec{x}_{ij}\beta\right)\right) ].&
\end{eqnarray*}
Denoting $\mu_{ij}=P(y_{ij}=1|X_i),$ and $\mu_i=(\mu_{i1},\dots,\mu_{im})^T,$ we have
\begin{eqnarray*}
&cl^{(1)}(\beta;Y) \ \  =  \ \ \sum_{i=1}^{n} \left(\frac{\partial \mu_i}{\partial_\beta}\right)^T \Pi_i^{-1}(y_i-\mu_i),&
\end{eqnarray*}
where $\Pi_i=\mbox{diag}(\mbox{var}(y_{i1}),\cdots,\mbox{var}(y_{im}))$, and $\mbox{var}(y_{ij})= \mu_{ij}(1-\mu_{ij})$.   This yields
\begin{eqnarray*}
&\hspace{-1cm}H(\beta) = n^{-1} \sum_{i=1}^n  \left(\frac{\partial \mu_i}{\partial_\beta}\right)^T  \Pi_i^{-1}\left(\frac{\partial \mu_i}{\partial_\beta}\right) \ \ \mbox{ and } \ \ J(\beta) = n^{-1} \sum_{i=1}^n  \left(\frac{\partial \mu_i}{\partial_\beta}\right)^T  \Pi_i^{-1} \, \mbox{cov}( y_i) \, \Pi_i^{-1} \left(\frac{\partial \mu_i}{\partial_\beta}\right).&
\end{eqnarray*}
To find the estimates $\widehat \beta_n^c,$ we use the Newton-Raphson algorithm. Denote 
$\widehat \mu_{in}=\{\widehat \mu_{i1n}, \widehat \mu_{i2n}, \ldots, \widehat \mu_{i{m}n}\}^T,$
where $\widehat{\mu}_i= \Phi(X_i\widehat\beta_n^c)$. Let $\widehat \Pi_{in}$ denote the estimator of $\Pi_i$ obtained by substituting $\widehat \mu_{ijn}$ for $\mu_{ij}.$  We estimate $H(\beta)$ and $J(\beta)$ as
\begin{eqnarray*}
&\hspace{-2.4cm}\widehat{H}_n \ \ = \ \ n^{-1}\sum_{i=1}^n  (\left.\frac{\partial \mu_i}{\partial_\beta}\right|_{\widehat \beta_n^c})^T  \widehat \Pi_{in}^{-1}(\left.\frac{\partial \mu_i}{\partial_\beta}\right|_{\widehat \beta_n^c})&\\
&\widehat{J}_n \ \ = \ \ n^{-1}\sum_{i=1}^{n} (\left.\frac{\partial \mu_i}{\partial_\beta}\right|_{\widehat \beta_n^c})^T\widehat \Pi_{in}^{-1} \ \ \widehat{\mbox{cov}}_n( y_i) \ \ \widehat \Pi_{in}^{-1}(\left.\frac{\partial \mu_i}{\partial_\beta}\right|_{\widehat \beta_n^c}),&
\end{eqnarray*}
calculating the empirical variance as $\widehat{\mbox{cov}}_n(y_i)= (y_i -\widehat{\mu}_{in})(y_i - \widehat{\mu}_{in})^T.$ }

\subsection{Quadratic Exponential Model}\label{sec:QE}
The quadratic exponential model is a popular tool used to model clustered binary data with intra-cluster interactions \citep{geysMole}.   In this model, the binary observations take values $y_{ij} \in \{-1,1\}$ and the joint distribution is given by
\begin{eqnarray}\label{line:QEmodel}
&f_{Y}(y_i) \ \ \propto \ \ \exp\left\{\sum_{j=1}^{m_i}\mu_{ij}^*y_{ij} + \sum_{j< j'}w_{ijj'}^*y_{ij}y_{ij'}\right\}, &
\end{eqnarray}
where $\mu_{ij}^*$ is a parameter which describes the main effect of the measurements and $w^*_{ijj'}$ describes the association between pairs of measurements within the cluster $y_i$.  Independence corresponds to the case that $w^*_{ijj'}=0$ and positive or negative correlation corresponds to $w^*_{ijj'}>0$ or $w^*_{ijj'}<0$, respectively.   For simplicity, we consider the case that  $\mu_{ij}^*=\mu_i^*$ and $w_{ijj'}^*=w_i^*$, noting that our methodology can be readily applied to the general scenario as well. Under this simplification, \citet{moleRyan}, showed that the joint distribution can be equivalently written in terms of $z_i=\sum_{j=1}^{m_i} 1(y_{ij}=1)$ (the number of successes in the $i$th cluster) as 
$\label{line:MRmodel}
f_{Y}(y_i) \propto \exp\{\mu_i z_i - w_iz_i(m_i - z_i)\},
$
where $w_i= 2w_i^*$ and $\mu_i= 2\mu_i^*.$ 

Specifying the normalizing constant in \eqref{line:QEmodel} is famously difficult, but also necessary to compute the full likelihood function.  It is therefore desirable to use an alternative approach, one which does not involve such an intensive calculation.  Replacing the joint distribution with the conditional distributions leads to a conditional composite likelihood function
$
cl(\mu, w;Y)=\sum_{i=1}^n \sum_{j=1}^{m_i} \log f(y_{ij}|\{y_{ij'}\},j'\neq j),
$
which does not require computation of the normalizing constant.  

We now define two conditional probabilities
\begin{eqnarray*}
&p_{is} = \frac{\exp\{\mu_i -w_i(m_i -2z_i +1)\}}{1+\exp\{\mu_i -w_i(m_i -2z_i +1)\}}, \ \ \  \ \ \ \ \ 
p_{if}  = \frac{\exp\{-\mu_i +w_i(m_i -2z_i -1)\}}{1+\exp\{-\mu_i +w_i(m_i -2z_i -1)\}}. &
\end{eqnarray*}
Heuristically, $p_{is}$ is the conditional probability of one more success, given $z_i-1$ successes and $m_i - z_i$ failures, while $p_{if}$ is the conditional probability of one more failure, given $z_i$ successes and $m_i - z_i-1$ failures.  Note that $p_{if} \neq 1-p_{is},$ because of the term $m_i-2z_i \pm 1.$  The composite likelihood can now be expressed as 
$\label{line:CLexp}
cl(\mu, w;Y) =\sum_{i=1}^n \left( z_i \log p_{is}+(m_i -z_i) \log p_{if} \right).
$

This special form of the composite likelihood means that a logistic regression approach can be used to estimate the parameters.  We model a covariate effect by using the linear model $\mu_i=X_i\beta$, with ${w_i=w}$ interpreted as an additional parameter.   That is, for the parameter $w,$ the value of the covariate is set to $-(m_i-2z_i+1)$ when $y_{ij}=1$ and $-(m_i-2z_i-1)$ when $y_{ij}=-1$.   This allows us to obtain CMLE estimates of both $\beta$ and $w$ using iterative re-weighted least squares, commonly used to solve logistic regression maximization problems.
To estimate the covariance of $\cmlebeta$, we computed $\widehat{J}_n$ as the empirical variance of the score vector, 

plugging in estimates of $\mu_i^*, w^*$ throughout.  The Hessian matrix $\widehat{H}_n$ is estimated using the result from fitting the logistic model in \proglang{R}, see \citet{geysMole}.

\section{Simulation Results}\label{sec:simulation}

\begin{table}[b!]
\caption{Multiple comparison methods considered:}\label{table:simstab}
\begin{center}
\begin{tabular}{ccc}
\toprule[1.5pt]
\multirow{2}{*}{CASE} &multiple comparison & \multirow{2}{*}{$\widehat \Gamma_n$} \\
 &method &  \\
\midrule[1.5pt]
(a) & MNQ &  \multirow{4}{*}{$\widehat H_n^{-1}\widehat J_n \widehat H_n^{-1}$}\\
(b) & Bonferroni &  \\
(c) & Dunn-Sid\'ak  &  \\
(d) & Holm  &  \\
(e) &Scheff\'e &\\ \midrule
(f) & MNQ ``naive" &  $\widehat H_n^{-1}$\\
\bottomrule[1.5pt]
\end{tabular}
\end{center}
\end{table}

We evaluate the validity of our proposed approach on three different multivariate models from Section 3 using simulations. We test two different global null hypotheses on the regression
coefficients  $\beta_1,\cdots,\beta_p$: many-to-one comparisons,
$H_{0}:\cap_{i=2}^p\{\beta_1=\beta_i\}$;  and all pairwise
comparisons $H_{0}:\cap_{1\leq i,j\leq p, i\neq j}
\{\beta_i=\beta_j\}$. The results for many-to-one comparisons are summarized here while the results for all pairwise comparisons are provided in the supplementary material.  We choose a collection of different types of multiple testing methods including one-step methods (Bonferroni and Dunn-Sid\'ak), a stagewise (Holm), a projection method (Scheff\'e), and the MNQ method based on the multivariate distribution of test statistics. For the MNQ method, the critical values can
be obtained using the \proglang{R} package \pkg{mvtnorm}
\citep{hoth}.

Part of our goal is to show practitioners what happens if the correlation structure in the clustered data is ignored.   To this end, we also include a ``misspecified" scenario, where independence is erroneously assumed within the clusters.    Due to the specific composite likelihood methods we use (univariate marginals and univariate
conditionals), such a misspecification is equivalent to  $H(\theta)=J(\beta)$ in~\eqref{line:godambe}.    This
results in an estimate of $\widehat \Gamma_n = \widehat H_n^{-1}$ in
Theorem~\ref{thm:limit}.   This misspecified scenario is included
for comparison, and we consider it only with the MNQ multiple
comparison method (that is, the MNQ cutoff is calculated based on
$V$ estimated by plugging in $\widehat \Gamma_n = \widehat
H_n^{-1}$).   Throughout, it is referred to as the ``naive"
approach.   Overall, we therefore consider six different
approaches, and these are provided in Table~\ref{table:simstab}.

{\indent In our simulations, we study the three models described in the previous section.   For each model, a different sample size is needed for our asymptotic approximations to be valid.  We determine this sample size with an initial simulation, before we proceed with our more in-depth investigations.   For each simulation setting,  $10\, 000$ simulated data sets were generated and the family-wise type I error rate was set to $0.05$. The standard deviation for the observed FWER is hence approximately $0.002$.   These preliminary simulation results are given in Table~\ref{tab:ens}.  We observe that $n=200, 500$ and $700$ are required for the multivariate normal, multivariate probit and quadratic exponential models to maintain FWER within two standard deviations away from 0.05, respectively.  These are the sample sizes used for the simulation results which follow. }

\begin{table}[h!]
\renewcommand{\arraystretch}{1.35}
\caption{FWER for different sample sizes}
\label{tab:ens}
\begin{center}
\begin{tabular}{@{}cccccc@{}} \hline
\toprule[1.5pt] \multirow{3}{*}{model} & &Sample size\\
\cmidrule{2-6}
      &  200 & 500&700& 1000   & 4000\\
\midrule[1.5pt]
multivariate normal                 &  0.0509   & 0.0492  & 0.0483 & 0.0495 & 0.050\\
multivariate probit                 &  0.0576   & 0.0501  & 0.0511 & 0.0506   &0.0511 \\
quadratic exponential               &  0.0580   & 0.0543  & 0.0519 & 0.0520  & 0.0504\\
 \bottomrule[1.5pt]
\end{tabular}
\end{center}
\end{table}

{To evaluate the power of each of the different methods, we consider two different alternative scenarios: one alternative configuration $a_1$ with only one non-zero parameter with a large effect size, and a second alternative configuration $a_2$ with five true non-zero parameters but with small effect sizes for all.   We are interested in the ability of the test to reject the global null hypothesis, but also in the ability of the test to reject the individual null hypotheses.   Under the alternative scenario $a_1,$ we calculate the power to reject the global hypothesis (denoted as ``$a_1$" in the tables) and  for the alternative configuration $a_2$, we calculate both the power to reject the global null hypothesis (denoted as ``$a_2$" in the tables) and the sum of the five powers to rejected the five individual true alternatives (denoted as ``ind $a_2$" in the tables).   Note also that the true effects are purposefully chosen to be small under $a_2,$ and therefore the typical empirical results are considerably smaller than $5,$ as expected.  This is done so that the observed global power is not uniformly high, which allows us to detect more subtle differences among the various methods.}

\subsection{Multivariate Normal Model}

{We consider the multivariate normal model with $n=200$ clusters, cluster size $m=4$ or $10$, and the number of covariates set to $p=10$ or $ 20.$  Four different $\Sigma$ scenarios are considered: 1) three exchangeable structures with $\sigma^2=0.8$ and $\rho=\mbox{cov}(y_{ij}, y_{ik})=0,$  $0.2$ or $0.5$; 2) one arbitrary structure, where $\Sigma=$ $((1.3,0.9,0.5,0.3)^T,(0.9,1.9,1.3,0.3)^T,$ $(0.5,1.3,1.3,0.1)^T,(0.3,0.9,0.1,0.7)^T).$ In each simulation, the $m\times p$ covariate matrix $X_i$ is obtained by randomly sampling from normal distributions. 

We consider here the many-to-one comparisons where the first parameter is taken as the baseline. Under the global null hypothesis $H_0,$ the true value of the regression parameters is set to $\beta^T=0,$ and the power is calculated under two different alternative configurations  $\beta_{a_1}^T=(0, 0, 0, 0.032, 0,\ldots, 0)$ and $\beta_{a_2}^T=(0, 0.008, 0.01, -0.03, 0.005, -0.01, 0,\ldots, 0)$. Under $\beta_{a_1}$, there is only one true alternative, and we evaluate the power to reject the global null hypothesis. Under $\beta_{a_2}$, there are five true alternatives and we evaluate both the power to reject the global null and the sum of five powers to reject the five true alternatives.  

\indent Table~\ref{tab:normal} (three exchangeable $\Sigma$ scenarios) and Table~\ref{tab:4} (general $\Sigma$) summarize the results of our simulations.  Overall, it is shown that the method which utilizes MNQ and  correctly accounts for the intra-cluster correlations, has the best performance.  A comparison of MNQ and naive MNQ clearly shows the cost of ignoring these correlations:  the FWER of MNQ is superior to that of naive MNQ for $\rho\neq0$ (when $\rho=0$ the two methods are almost identical).   Notably, the power of the naive MNQ is occasionally higher than that of MNQ, however, this is only due to the over-inflation of the naive MNQ's FWER.  Overall, MNQ exhibits the best performance among all of the multiple comparison procedures.  The small effect sizes chosen under $a_2$ allow us to detect more subtle differences in the performance of the methods.  Notice that for the rejection of the global null hypothesis, Holm's method has exactly the same power as that of the Bonferroni method. However, for the individual powers, Holm's method has higher power to reject individual hypothesis than the Bonferroni method.  

\indent We also evaluate the efficiency of the maximum composite likelihood estimator versus maximum likelihood estimator.   That is, we compute the ratio of the standard error of the MLE versus that of the MCLE.  For small $\rho$, the ratio is close to one and as $\rho$ increases, the ratio decreases. This demonstrates that the efficiency of composite likelihood estimator decreases with the increase of the intra-cluster correlation, as expected. }

\begin{table}[H]
\scriptsize
\caption{Simulations results for the multivariate normal model with exchangeable $\Sigma$ }
\label{tab:normal}
\begin{center}
\begin{tabular}{@{}cccccccccccc@{}} \hline
\toprule[1.5pt]   &  $\rho$ & $m$ &   $p$  &    MNQ & naive& Bonf & S-D &Holm & Scheff\'e & efficiency \\
\midrule[1.5pt]
FWER &\multirow{16}{*}{0} & \multirow{8}{*}{4} & \multirow{4}{*}{10} &  0.0545& 0.0553& 0.0419& 0.0427& 0.0419& 0.0007&  0.9983\\
 $a_1$& & & &\g 0.8164&\g 0.8166&\g 0.7894&\g 0.7918& \g 0.7894&\g 0.2801\\
$a_2$ & &  & &\g 0.8057&\g 0.8053&\g 0.7617&\g 0.7738& \g 0.7617&\g 0.2226 \\
 ind $a_2$ & & & &\g 0.9080&\g 0.9079&\g 0.8417&\g 0.8503&\g 0.8848& \g 0.2242 \\
\cmidrule{4-11}
FWER&& & \multirow{4}{*}{20}  & 0.0511& 0.0502& 0.0352& 0.0363& 0.0352& 0.0000& 0.9980\\
$a_1$& & & & \g 0.7487&\g 0.7476&\g 0.7062&\g 0.7086& \g 0.7062&\g  0.0259  \\
$a_2$ & &  & &\g 0.7150&\g 0.7134&\g 0.6687&\g 0.6628& \g 0.6687&\g 0.0162 \\
ind $a_2$ & & & &\g 0.7698&\g 0.7674&\g 0.7081&\g 0.7007& \g 0.7518&\g
0.0162 \\ \cmidrule{3-11}
FWER &&\multirow{8}{*}{10} & \multirow{4}{*}{10}  & 0.0479& 0.0471& 0.0375& 0.0378& 0.0375& 0.0001&  0.9989 \\
$a_1$& & & &\g 0.9983&\g 0.9983&\g 0.9979&\g 0.9980& \g 0.9979&\g 0.9284\\
$a_2$ & &  & &\g 0.9993&\g 0.9993&\g 0.9986&\g 0.9990& \g 0.9986&\g 0.8792\\
ind $a_2$ && &  &\g 1.4822&\g 1.4816&\g 1.4219&\g 1.4284& \g 1.4896 &\g 0.8933\\
\cmidrule{4-11}
FWER & & &   \multirow{4}{*}{20} &0.0487& 0.0485& 0.0363& 0.0373& 0.0363& 0.0000&  0.9986\\
$a_1$& & & &\g 0.9981&\g 0.9980&\g 0.9967&\g 0.9969& \g 0.9967&\g 0.5428\\
$a_2$ & &  & &\g 0.9978&\g 0.9977&\g
0.9963&\g 0.9957& \g 0.9963&\g 0.4137 \\
ind $a_2$ & & & &\g 1.3439&\g 1.3406&\g 1.2759&\g 1.2776& \g 1.3267&\g 0.4139\\  \midrule[1.5pt]
FWER &\multirow{16}{*}{0.2} &\multirow{8}{*}{4} & \multirow{4}{*}{10}  &  0.0494& 0.0670&0.0389& 0.0397& 0.0389& 0.0001& 0.9453 \\
$a_1$& & & & \g 0.7760& \g 0.8113&\g 0.7453&\g 0.7476& \g 0.7453&\g 0.2481\\
$a_2$& & & &\g 0.7630&\g 0.8044&\g 0.7224&\g 0.7280& \g 0.7224&\g 0.1831  \\
ind $a_2$& && &\g 0.8556&\g 0.9268&\g 0.7939&\g 0.8032& \g 0.8317&\g 0.1845 \\
\cmidrule{4-11}
FWER && &  \multirow{4}{*}{20} &  0.0533& 0.0734&0.0390& 0.0397& 0.0390& 0.0000&  0.9430\\
$a_1$& & & &\g 0.7044&\g 0.7490&\g 0.6591&\g 0.6617& \g 0.6591&\g 0.0191 \\
$a_2$ & &  & &\g 0.6713&\g 0.7200&\g 0.6187&\g 0.6148& \g 0.6187&\g 0.0102\\
ind $a_2$&&& &\g 0.7106&\g 0.7777&\g 0.6476&\g 0.6438& \g 0.6937&\g 0.0102\\ \cmidrule{3-11}
FWER&& \multirow{8}{*}{10} & \multirow{4}{*}{10}  &  0.0467& 0.1019&0.0357& 0.0365& 0.0357& 0.0003&  0.8685\\
$a_1$& & & &\g 0.9912&\g 0.9974&\g 0.9875&\g 0.9880& \g 0.9875&\g 0.8098 \\
$a_2$ & &  & &\g  0.9925&\g 0.9983&\g 0.9877&\g 0.9897& \g 0.9877&\g 0.7295 \\
ind $a_2$&& & &\g  1.3506&\g 1.5795&\g 1.2871&\g 1.2968& \g 1.3407&\g 0.7374 \\
\cmidrule{4-11}
FWER&&  &  \multirow{4}{*}{20} &  0.0468& 0.1057&0.0320& 0.0331& 0.0320& 0.0000 &0.8636\\
$a_1$& & & &\g 0.9868&\g 0.9970&\g 0.9813&\g 0.9819&\g 0.9813&\g 0.3114\\
$a_2$ & &  & &\g 0.9820&\g 0.9964&\g 0.9758&\g 0.9752& \g 0.9758&\g 0.2146\\
ind $a_2$& && &\g 1.2100&\g 1.4205&\g 1.1495&\g 1.1545& \g 1.1891&\g 0.2146\\
\midrule[1.5pt]
FWER&\multirow{16}{*}{0.5} & \multirow{8}{*}{4} & \multirow{4}{*}{10}  &  0.0513& 0.0977&0.0390& 0.0398& 0.0390& 0.0007 & 0.7491\\
$a_1$& & & & \g 0.7235&\g 0.8129&\g 0.6867&\g 0.6904& \g 0.6867&\g 0.1947 \\
$a_2$ & &  & &\g 0.6922&\g 0.8042&\g 0.6615&\g 0.6571& \g 0.6615&\g 0.1497 \\
ind $a_2$& &&  &\g 0.7570&\g 0.9391&\g 0.7208&\g 0.7074 &\g 0.7533&\g 0.1502 \\\cmidrule{4-11}
FWER & & &  \multirow{4}{*}{20}  & 0.0510& 0.1029&0.0377& 0.0385& 0.0377& 0.0000 &0.7343\\
$a_1$& & & & \g 0.6420& \g 0.7526&\g 0.5950& \g 0.5985& \g 0.5950&\g  0.0140\\
$a_2$ & &  & &\g 0.6031&\g 0.7322&\g 0.5437&\g 0.5508& \g 0.5437&\g 0.0076  \\
ind $a_2$& && &\g 0.6369&\g 0.8035&\g 0.5677&\g 0.5750  &\g 0.6109&\g 0.0076  \\
\cmidrule{3-11}
FWER& &\multirow{8}{*}{10} & \multirow{4}{*}{10}  & 0.0520& 0.2079&0.0410& 0.0417& 0.0410& 0.0000&  0.6070\\
$a_1$& & & & \g 0.9570&\g 0.9936&\g 0.9466&\g 0.9469&\g 0.9466&\g 0.6125 \\
$a_2$ & &  & &\g 0.9555&\g 0.9982&\g 0.9438&\g 0.9431& \g 0.9438&\g 0.5062 \\
ind $a_2$& &&  &\g 1.1914&\g 1.6903&\g 1.1367& \g 1.1372&\g 1.1877 &\g 0.5096   \\ \cmidrule{4-11}
FWER& & &  \multirow{4}{*}{20} & 0.0459& 0.2271&0.0328& 0.0337& 0.0328& 0.0000&0.5898 \\
$a_1$& & & & \g 0.9403& \g 0.9938&\g 0.9224& \g 0.9243& \g 0.9224& \g 0.1408  \\
$a_2$ & &  & &\g 0.9222&\g 0.9948&\g 0.8932&\g 0.8968& \g 0.8932&\g 0.0871 \\
ind $a_2$& &&  &\g 1.0589&\g 1.5362&\g 0.9907&\g 0.9983 &\g 1.0306 &\g 0.0871 \\
 \bottomrule[1.5pt]
\end{tabular}
\end{center}
\end{table}

\begin{table}[H]
\caption{Simulations results for the multivariate normal model with unstructured $\Sigma$} \label{tab:4}
\begin{center}
\begin{tabular}{@{}cccccccccc@{}} \hline
\toprule[1.5pt]   &   $m$ &   $p$  &    MNQ & naive& Bonf & S-D &Holm & Scheff\'e \\
\midrule[1.5pt]
 FWER& \multirow{8}{*}4 & \multirow{4}{*}{10}  & 0.0464& 0.0729&0.0345& 0.0358& 0.0345& 0.0004&\\
 $a_1$& & &  \g 0.6348&\g 0.7089&\g 0.5962&\g 0.5992&\g 0.5962&\g 0.1358  \\
 $a_2$ & &  & \g 0.5123&\g 0.6045&\g 0.4763&\g 0.4714& \g 0.4763&\g 0.0614\\
 ind $a_2$& & & \g 0.5499&\g 0.6707&\g 0.5094&\g 0.5112 &\g 0.5357&\g 0.0615\\\cmidrule{3-10}
FWER& & \multirow{4}{*}{20}  &0.0390& 0.0664&0.0285& 0.0290& 0.0285& 0.0000\\
 $a_1$& &    &\g 0.5205&\g 0.6081&\g 0.4694&\g 0.4736&\g 0.4694&\g 0.0046   \\
$a_2$ & &  & \g 0.3913&\g 0.4864&\g 0.3378&\g 0.3428& \g 0.3378&\g 0.0011\\
ind $a_2$& && \g 0.4057&\g 0.5090&\g 0.3436&\g 0.3508&\g 0.3743&\g 0.0011 \\
\cmidrule{2-10}
FWER & \multirow{8}{*}{10} & \multirow{4}{*}{10} & 0.0472& 0.0407&0.0360& 0.0367& 0.0360& 0.0004\\
 $a_1$& &  & \g 0.6310&\g 0.6102&\g 0.5906&\g 0.5940&\g 0.5906&\g  0.1198 \\
 $a_2$ & &  & \g 0.5025&\g 0.4779&\g 0.4560&\g 0.4599& \g 0.4560&\g 0.0537\\
 ind $a_2$& && \g 0.5442&\g 0.5142&\g 0.4870 &\g 0.4911 &\g 0.5158&\g 0.0537\\ \cmidrule{3-10}
FWER  &  &  \multirow{4}{*}{20} &0.0361& 0.0302&0.0262& 0.0267& 0.0262& 0.0000\\
 $a_1$& &    &\g 0.5078&\g 0.4865&\g 0.4585&\g 0.4615& \g 0.4585&\g 0.0025 \\
$a_2$ & &   & \g 0.3668&\g 0.3448&\g 0.3148&\g 0.3167& \g 0.3148&\g 0.0010 \\
ind $a_2$& & & \g 0.3711&\g 0.3490&\g 0.3186& \g 0.3204&\g 0.3462&\g 0.0010  \\
 \bottomrule[1.5pt]
\end{tabular}
\end{center}
\end{table}

\subsection{Multivariate Probit Model}

{Here, we consider $n=500$ clusters with a cluster size $m=4,$ or $10.$ The binary variables are generated by dichotomizing latent multivariate normal variables with a threshold of zero. For each cluster, an $m\times p$ covariate matrix $X_i$, with $p=10$ or $20,$ is obtained by randomly sampling from normal distributions.  The regression coefficients under the global null hypothesis is $\beta^T= 0$ and the two alternative configurations are $\beta_{a1}^T=(0, 0, 0, 0.03, 0,\ldots, 0)$ and $\beta_{a2}^T=(0, 0.008, 0.01, -0.03, 0.005, -0.01, 0,\ldots, 0)$. The latent multivariate random vector has a mean $X_i\beta$ and a correlation matrix with $\rho$ on the off-diagonals and $\sigma=1$. Here, we consider $\rho=0,$ or $0.5.$ 

\indent The empirical results are given in Table~\ref{tab:5}.   The results show that the MNQ method has overall the best performance.  We note though that for the two settings when $\rho=0.5$ and $p=20,$ the MNQ method has FWER more than 2 standard deviations away from 0.05.  
Similarly to the multivariate normal setting, the naive MNQ for the multivariate probit model has large FWER when $\rho=0.5.$  For the global hypothesis, the Sid\'ak method has higher power than that of the Bonferroni and Holm method, whereas the Holm method has higher power to reject individual null hypotheses than the Bonferroni and Sid\'ak method.}  

\subsection{Quadratic Exponential Model}

{Here, we take a total of $n=700$ clusters, and $p=10$ or $20$ predictors. The number of observations within each clusters, $m_i,$ varies between clusters and is uniformly sampled from $\{4,5,6,7,8\}$. The $m_i\times p$ covariate matrix $X_i$ is  sampled from a standard normal distribution. We also consider two different values for the interaction parameter: $w=0$ or $0.5$. The null value of the regression coefficients is $\beta^T\equiv0$ and the two alternative configurations are to $\beta_{a1}^T=(0, 0, 0, 0.12, 0,\ldots, 0)$ and $\beta_{a2}^T=(0, 0.08, 0.12, -0.03 ,0.05 , -0.08, 0,\ldots, 0)$. The empirical FWER and power are computed and summarized in Table~\ref{tab:6}.  Overall, MNQ has clearly the best performance.}

\begin{table}[H]
\caption{Simulation results for the probit model} \label{tab:5}
\begin{center}
\begin{tabular}{@{}cccccccccc@{}} \hline
\toprule[1.5pt]   &  $\rho$ & $m$ &   $p$  &    MNQ & naive& Bonf & S-D &Holm & Scheff\'e \\
\midrule[1.5pt]
FWER&\multirow{16}{*}{0}& \multirow{8}{*}{4} & \multirow{4}{*}{10}  &0.0530& 0.0506&0.0413& 0.0424& 0.0413& 0.0001\\
 $a_1$& & & &\g 0.8700&\g 0.8705&\g 0.8477&\g 0.8496& \g 0.8477&\g 0.3420 \\
$a_2$ & &&  &\g 0.9114&\g 0.9109&\g 0.8885&\g 0.8907& \g 0.8885&\g 0.3572\\
ind $a_2$& && &\g 1.0828&\g 1.0779&\g 1.0193&\g 1.0305 &\g 1.0682&\g 0.3590\\\cmidrule{4-10}
FWER& & & \multirow{4}{*}{20}& 0.0528& 0.0503&0.0389& 0.0395& 0.0389& 0.0000\\
 $a_1$& & &&\g 0.8258&\g 0.8232&\g 0.7902&\g 0.7924& \g 0.7902&\g 0.0460\\
$a_2$ & &&  &\g  0.8547&\g 0.8511&\g 0.8149&\g 0.8159& \g 0.8149&\g 0.0410  \\
ind $a_2$& && &\g  0.9436&\g 0.9389&\g 0.8847&\g 0.8825 &\g 0.9308&\g 0.0410  \\ \cmidrule{3-10}
FWER&& \multirow{8}{*}{10} & \multirow{4}{*}{10}  &0.0526& 0.0515&0.0423& 0.0428& 0.0423& 0.0005\\
 $a_1$& & &&\g 0.9996&\g 0.9996&\g 0.9995&\g 0.9995& \g 0.9995&\g 0.9641 \\
$a_2$ & &&  &\g 1.0000& \g 1.0000&\g 1.0000&\g 1.0000& \g 1.0000&\g 0.9695  \\
ind $a_2$& && &\g 1.6649& \g 1.6594&\g 1.5839& \g 1.5939&\g 1.6658&\g 1.0024   \\ \cmidrule{4-10}
FWER& & & \multirow{4}{*}{20}&0.0527& 0.0508 &0.0364& 0.0375& 0.0364& 0.0000\\
 $a_1$& & &&\g 0.9993&\g 0.9995&\g 0.9985&\g 0.9985& \g 0.9985&\g 0.6596 \\
$a_2$ & &&  &\g 1.0000&\g 0.9999&\g 0.9997&\g  0.9997& \g 0.9997&\g 0.6603 \\
ind $a_2$& && &\g 1.4867&\g 1.4780&\g  1.4057&\g  1.4062 &\g 1.4624&\g 0.6607 \\
\midrule[1.5pt]
FWER&\multirow{16}{*}{0.5} & \multirow{8}{*}{4}& \multirow{4}{*}{10} &0.0508& 0.0793&0.0393& 0.0404& 0.0393& 0.0003\\
 $a_1$& & & &\g 0.8102&\g 0.8601&\g 0.7808&\g 0.7841& \g 0.7808&\g 0.2726 \\
$a_2$ & &&  &\g 0.8530&\g 0.9038&\g 0.8305&\g 0.8258& \g 0.8305&\g 0.2689   \\
ind $a_2$& &&  &\g 0.9852&\g 1.1028&\g 0.9321&\g 0.9334 &\g 0.9768&\g 0.2708   \\ \cmidrule{4-10}
FWER& & & \multirow{4}{*}{20}  &0.0585& 0.0915&0.0406& 0.0415& 0.0406& 0.0000\\
 $a_1$& & & &\g 0.7578&\g 0.8196&\g 0.7082&\g 0.7106& \g 0.7082&\g 0.0264 \\
 $a_2$ & & &  &\g 0.7891&\g 0.8534&\g  0.7365&\g 0.7428& \g  0.7365&\g 0.0247   \\
ind $a_2$& &&  &\g 0.8637&\g 0.9712&\g  0.7855 &\g 0.7963&\g 0.8330&\g 0.0247   \\  \cmidrule{3-10}
FWER& & \multirow{8}{*}{10}& \multirow{4}{*}{10}&  0.0513& 0.1437&0.0402& 0.0412& 0.0402& 0.0005\\
 $a_1$& & & &\g 0.9900&\g 0.9979&\g 0.9871&\g 0.9876& \g 0.9871&\g 0.8017 \\
$a_2$ & &&  &\g 0.9952&\g 0.9997&\g 0.9966&\g 0.9939& \g 0.9966&\g 0.8038  \\
ind $a_2$& &&  &\g 1.4075&\g 1.7926&\g 1.3520&\g 1.3552 &\g 1.4154&\g 0.8147  \\ \cmidrule{4-10}
FWER& & & \multirow{4}{*}{20}& 0.0543& 0.1622&0.0382& 0.0389& 0.0382& 0.0000\\
 $a_1$& & & &\g 0.9862&\g 0.9974&\g 0.9784&\g 0.9787& \g 0.9784&\g 0.3081 \\
$a_2$ & &&  &\g 0.9935&\g 0.9998&\g 0.9894&\g 0.9883& \g 0.9894&\g 0.3006  \\
ind $a_2$& &&  &\g 1.2873&\g 1.6251&\g 1.2248&\g 1.2218 &\g 1.2777&\g 0.3006  \\
 \bottomrule[1.5pt]
\end{tabular}
\end{center}
\end{table}

\begin{table}[H]
\caption{Simulation results for the quadratic exponential model} \label{tab:6}
\begin{center}
\begin{tabular}{@{}ccccccccccc@{}}
\toprule[1.5pt]   &  $w$ & $m$ &   $p$  &    MNQ & naive& Bonf & S-D &Holm & Scheff\'e \\
\midrule[1.5pt]
FWER&\multirow{16}{*}{0}& \multirow{8}{*}{4} & \multirow{4}{*}{10}  &0.0514& 0.0562&0.0400& 0.0403& 0.0400& 0.0001\\
 $a_1$& & &&\g 0.5390&\g 0.5534&\g 0.5010&\g 0.5046& \g 0.5010&\g 0.0777\\
 $a_2$ & & & &\g 0.7067&\g 0.7240&\g 0.6573&\g 0.6636& \g 0.6573&\g 0.0935 \\
 ind $a_2$& &&   &\g 0.9779&\g 1.0283&\g  0.8826 &\g 0.8888 &\g 0.9373&\g 0.0986   \\ \cmidrule{4-10}
FWER& & & \multirow{4}{*}{20}& 0.0561& 0.0767&0.0404& 0.0412& 0.0404& 0.0000\\
 $a_1$& & & &\g 0.4551&\g 0.4853&\g 0.3990&\g 0.4021& \g 0.3990&\g 0.0025 \\
 $a_2$ & & & &\g 0.6040&\g 0.6365&\g 0.5237&\g 0.5403& \g 0.5237&\g 0.0027 \\
ind $a_2$& &&   &\g 0.7731&\g 0.8347&\g 0.6514 &\g 0.6679 &\g 0.7029&\g 0.0027   \\  \cmidrule{3-10}
FWER&& \multirow{8}{*}{10} & \multirow{4}{*}{10}  &0.0491& 0.0549&0.0381& 0.0384& 0.0381& 0.0001\\
 $a_1$& & &&\g 0.5391&\g 0.5535&\g 0.5010&\g 0.5046& \g 0.5010&\g 0.0779 \\
 $a_2$ & & & &\g 0.7066&\g 0.7239&\g 0.6573&\g 0.6636& \g 0.6573&\g 0.0934 \\
ind $a_2$& &&   &\g 0.9780&\g 1.0284&\g  0.8826&\g  0.8890&\g 0.9373&\g 0.0985    \\ \cmidrule{4-10}
FWER& & & \multirow{4}{*}{20}& 0.0561& 0.0767&0.0404& 0.0412& 0.0404& 0.0000\\
 $a_1$& & & &\g 0.4548&\g 0.4849&\g 0.3989&\g 0.4020& \g 0.3989&\g 0.0026\\
$a_2$ & & & &\g  0.5971&\g 0.6309&\g 0.5255&\g 0.5361& \g 0.5255&\g 0.0013 \\
ind $a_2$& &&   &\g  0.7681&\g 0.8316&\g 0.6527&\g 0.6688 &\g 0.7043 &\g 0.0013 \\
\midrule[1.5pt]
FWER&\multirow{16}{*}{0.5} & \multirow{8}{*}{4}& \multirow{4}{*}{10}  &0.0521& 0.0000&0.0417& 0.0424& 0.0417& 0.0002\\
 $a_1$& & &&\g 0.7864&\g 0.0307&\g 0.7546&\g 0.7582 & \g 0.7546&\g 0.2329\\
 $a_2$ & & & &\g 0.9050&\g 0.0444&\g 0.8800&\g 0.8772& \g 0.8800&\g 0.2531    \\
ind $a_2$ & & & &\g 1.5136&\g 0.0452&\g 1.4102&\g 1.4089 &\g 1.4915 &\g 0.2753    \\ \cmidrule{4-10}
FWER& & & \multirow{4}{*}{20}  & 0.0509& 0.0000&0.0377& 0.0383& 0.0377& 0.0000\\
 $a_1$& & & &\g 0.7214&\g 0.0158&\g 0.6739&\g 0.6769& \g 0.6739&\g 0.0178\\
 $a_2$ & & & &\g 0.8460&\g 0.0148&\g 0.7998&\g 0.7976& \g 0.7998&\g 0.0132  \\
ind $a_2$ & & & &\g 1.2902&\g 0.0150&\g 1.1532&\g 1.1612 &\g 1.2141&\g 0.0134    \\ \cmidrule{3-10}
FWER& & \multirow{8}{*}{10}& \multirow{4}{*}{10} &0.0521& 0.0000&0.0417& 0.0424& 0.0417& 0.0002 \\
 $a_1$& & & &\g 0.7864&\g 0.0307&\g 0.7546&\g 0.7582& \g 0.7546&\g 0.2329\\
 $a_2$ & & & &\g 0.9141&\g 0.0407&\g 0.8800&\g 0.8855& \g 0.8800&\g 0.2518   \\
ind $a_2$ & & & &\g 1.5326&\g 0.0416&\g 1.4102&\g 1.4261 &\g 1.4915 &\g 0.2746   \\  \cmidrule{4-10}
FWER&& & \multirow{4}{*}{20}& 0.0509& 0.0000&0.0378& 0.0384& 0.0378& 0.0000\\
 a& & & &\g 0.7202&\g 0.0161&\g 0.6731&\g 0.6760& \g 0.6731&\g 0.0178 \\
 $a_2$ & & & &\g 0.8460&\g 0.0148&\g 0.7998&\g 0.7976& \g 0.7998&\g 0.0132  \\
ind  $a_2$ & & &&\g 1.2902&\g 0.0150&\g 1.1532& \g 1.1612 &\g 1.2141&\g 0.0134   \\
 \bottomrule[1.5pt]
\end{tabular}
\end{center}
\end{table}

\section{Analysis of Depression Data}\label{sec:data}

\begin{table}[H]
\renewcommand{\arraystretch}{1.35}
\caption{Composite likelihood estimates of the health factors' regression coefficients}
\label{tab:7}
\begin{center}
\begin{tabular}{rccc} \hline
           &   estimate & SE &  $p$-value    \\ \hline
sleeplessness     &  1.3330  &0.0290 &  $< 2e-16$ \\
smoking    &    0.2826 & 0.0439 &  $< 2e-16$ \\
high blood pressure     &    0.0764 &  0.0219 & $2.07e-11$ \\
diabetes    & 0.0710 & 0.0296 &  $8.96e-07$ \\
difficulty in walking      &   0.0695 & 0.0054 &   $< 2e-16$ \\
age       &   0.0007 &  0.00003 &  $< 2e-16$ \\
activity   & -0.0156 & 0.0064 & $2.35e-05$ \\
$w$         &   0.2877 & 0.0094 &  $< 2e-16$ \\ \hline
\end{tabular}
\end{center}
\end{table}

{We apply our proposed method to  the health and retirement study (HRS) dataset.  Information about health, financial situation, family structure, and health factors were collected by the RAND center at the University of Michigan. We perform multiple comparisons on the effects of seven health factors on depression status of seniors. 
Depression status  is recorded as a binary response variable, whereas the seven health factors include age (in months), smoking, restless sleep, diabetes, high blood pressure, frequent vigorous physical activity, and difficulty in walking. For each individual we include only the years for which all of the factors were recorded. In total,  there are $33\,636$ people included in the analysis and the number of repeated measurements vary across individuals. As the response variable is binary and the cluster sizes vary, the quadratic exponential model is a natural choice to model this data set.  

\indent The full likelihood approach is very computationally challenging for this model, and hence we use the proposed composite likelihood based method to perform inference.   The $w$ parameter in the quadratic exponential model allows us to account for the interaction effect among the repeated measurements for the same individuals. The MCLE estimates and the associated standard errors are reported in Table~\ref{tab:7}. 

To compare the effect sizes of all the seven health factors, we perform all pairwise comparisons on the seven parameters with $H_{0, i,j} = \{\beta_i = \beta_j\}$ for a total of $21$ null hypotheses.   Of the six methods described in Table~\ref{table:simstab}, we choose the MNQ approach based on its superior performance in Section~\ref{sec:simulation}.   To show how different the results will be if the within-patient correlations are ignored, we also compare the result of the MNQ with the naive MNQ method.  Both the MNQ and the naive MNQ reject the global null hypothesis that all pairs of health factors have equal effect on the depression status.   The results for the individual hypotheses are given in Table~\ref{tab:8}.  }

{The MNQ method rejects $15$ hypotheses, whereas the naive MNQ method rejects $18$ out of the total $21$ hypotheses. Based on the estimates of the effect sizes (Table~\ref{tab:7}), we note that restless sleep and smoking are the two health factors with the largest effect sizes. Both MNQ and naive MNQ reject the pairwise comparisons between restless sleep with all other health factors and smoking with all other factors. This shows that restless sleep and smoking are the two leading health factors for the occurrence of depression. High blood pressure, diabetes, and difficulty in walking have the third, fourth and fifth largest estimated effect sizes. When we examine the three pairwise comparisons among these three factors, both MNQ and naive MNQ accept the three null hypotheses, indicating that these three health factors have similar effects and importance to the disease. Furthermore, when we compare high blood pressure with age and activity, both methods reject the two comparisons, indicating that high blood pressure is more important than age and activity with regard to the disease development.  

MNQ and naive MNQ are in agreement in all the aforementioned comparisons. However, when we compare the effect sizes between age and diabetes, diabetes and activity, age and activity,  the MNQ method accepts these three null hypotheses while the naive method rejects all three. The difference between the two methods is due to the correlation among the repeated measurements, which is estimated as $\widehat w=0.285.$  By ignoring this correlation, as in the naive method, the standard errors are underestimated leading to more rejections.}

\begin{table}\label{table:example}
\renewcommand{\arraystretch}{1.35}
\caption{Results of MNQ and naive MNQ in testing individual null hypotheses in the depression study data set. A: fail to reject, R: reject $H_0$ }
\label{tab:8}
\begin{center}
\begin{tabular}{cccccccc} \hline
 $H_0$   &  MNQ        &   naive     & $H_0$&  MNQ        &   naive \\ \hline
$\beta_{sleep}=\beta_{smoke}$ &  R& R&$\beta_{hbp}=\beta_{diabet}$ &A   &  A\\  
$\beta_{sleep}=\beta_{hbp}$ &  R&R & $\beta_{hbp}=\beta_{dif\hspace{-0.05cm}f\ walk}$ & A  & A\\
$\beta_{sleep}=\beta_{diabet}$ &R  & R&  $\beta_{hbp}=\beta_{age}$ &R & R\\
$\beta_{sleep}=\beta_{dif\hspace{-0.05cm}f\ walk}$ & R  &R&    $\beta_{hbp}=\beta_{activity}$ & R &  R  \\
$\beta_{sleep}=\beta_{age}$ & R  &   R  & $\beta_{diabet}=\beta_{dif\hspace{-0.05cm}f\ walk}$ &A   & A\\
$\beta_{sleep}=\beta_{activity}$ & R  & R& $\beta_{diabet}=\beta_{age}$ & A &   R&\\
  $\beta_{smoke}=\beta_{hbp}$ &R &R&       $\beta_{diabet}=\beta_{activity}$ & A  &R\\
 $\beta_{smoke}=\beta_{diabet}$ &R  & R &     $\beta_{dif\hspace{-0.05cm}f\ walk}=\beta_{age}$ &R &R\\
  $\beta_{smoke}=\beta_{dif\hspace{-0.05cm}f\ walk}$ & R  &R&     $\beta_{dif\hspace{-0.05cm}f\ walk}=\beta_{activity}$ &R& R \\ 
   $\beta_{smoke}=\beta_{age} $ &R  & R& $\beta_{age}=\beta_{activity}$ &A &   R\\
 $\beta_{smoke}=\beta_{activity}$&R &R&\\ \hline
\end{tabular}
\end{center}
\end{table}

\section{Discussion}\label{sec:discussion}
In many correlated multivariate models, it is often difficult to perform multiple comparisons based on the full likelihood.  In this paper, we propose to use the composite likelihood method to construct multiple comparison procedures to overcome this computational difficulty.  Theory is developed based on the asymptotic properties of the composite likelihood test statistic and illustrated for three different models: multivariate normal, multivariate probit and quadratic exponential. The simultaneous quantile of multivariate normal is used as a threshold for test statistics compared to some well-known traditional thresholds. This MNQ method, which is based on composite likelihood test statistics and uses multivariate normal quantiles to derive cut-off values for the test statistics,  possesses a more acceptable family-wise type I error rate in most simulation settings, compared to the other test procedures.

\section{Acknowledgment}

All authors are grateful to the anonymous referees, Editor, and Associate Editor for their careful reading of the manuscript and valuable comments, which greatly improved the manuscript.   This work is supported by NSERC grants held by Gao and Jankowski.

\section*{Appendix}

\subsection{Simulation result on all pairwise comparisons}
In Section 4 of the main manuscript, we provide simulation results under various settings for many-to-one comparisons. Here, we provide additional simulations for all pairwise comparisons.  We simulate multivariate normal, multiviate probit and quadratic exponential models as described in Section 4 of the paper.  The global null hypothesis, sample size and the two alternative configurations are the same as those used in many-to-one comparisons.  We perform all pairwise comparisons with $m=4$ and $p=10.$ For the multivariate normal, multivariate probit, and quadratic exponential models, we consider $\rho=0,$ or $0.5.$  For the quadratic exponential, we consider $w=0, 0.5$.  The results are summarized in Table 1.  We again observe that the MNQ approach has the best performance. MNQ maintains good control of the FWER except for the case of quadratic exponential model with $\rho=0.5,$ where it is slightly above $0.05$. The naive MNQ either has either very large FWER or very small FWER, indicating its poor control of the error rate. Among all the methods which maintain good control of FWER, the MNQ method achieves the highest power.  In addition, we consider the Tukey approach, as it is a commonly  used testing procedure in all pairwise comparisons.

\begin{table}[t!]
\caption{Simulations results for the multivariate normal, probit, and quadratic exponential models}\label{tab:all}
\begin{center}
\begin{tabular}{@{}ccccccccc@{}} \hline
\toprule[1.5pt] model &  &  $\rho$  &    MNQ & naive& Bonf & S-D  & Scheff\'e& Tukey \\ \midrule[1.5pt]
\multirow{6}{*}{normal} &FWER& \multirow{3}{*}{0} &0.0537& 0.0562  & 0.0411& 0.0420  & 0.0038 &0.0536 \\
& a1& &\g 0.9274&\g 0.9266&\g 0.9096&\g 0.9113& \g 0.6115&\g 0.9256\\
& a2&  &\g 0.9800&\g 0.9807&\g 0.9735&\g 0.9740& \g  0.8173&\g 0.9792\\  \cmidrule{2-9}
&FWER & \multirow{3}{*}{0.5}   &   0.0484& 0.1101 &0.0358& 0.0365 &    0.0032 &  0.0489\\
&a1&  &\g 0.8611&\g 0.9245&\g 0.8325&\g 0.8346& \g  0.4769&\g 0.8587\\
& a2&  &\g 0.9492&\g 0.9775&\g 0.9346&\g  0.9361& \g  0.6854&\g 0.9482 \\ \hline 
\multirow{6}{*}{probit}&FWER&\multirow{3}{*}0  & 0.0534& 0.0494&0.0409& 0.0412&  0.0026& 0.0524\\
& a1&  &\g 0.9792& \g 0.9790&\g 0.9745&\g 0.9747& \g  0.7972&\g 0.9791\\
& a2 &  &\g 0.9961&\g 0.9961&\g 0.9946&\g 0.9946& \g  0.9321&\g 0.9959  \\\cmidrule{2-9}
&FWER&\multirow{3}{*}{0.5}  & 0.0523& 0.0864&0.0394& 0.0394& 0.0023& 0.0514\\
& a1&  &\g 0.9586&\g 0.9754&\g 0.9467&\g 0.9484& \g  0.6991&\g 0.9577  \\
&a2&  &\g  0.9885&\g 0.9938&\g 0.9842&\g 0.9848& \g  0.8707&\g 0.9884 \\ \hline 
\multirow{6}{*}{quad. exp.}&FWER&\multirow{3}{*}0  & 0.0534& 0.0631&0.0399& 0.0407&  0.0018& 0.0530\\
& a1& &\g  0.7710&\g 0.7869&\g 0.7270&\g 0.7301& \g  0.3224&\g 0.7678 \\
& a2&  &\g 0.9706&\g 0.9741&\g 0.9613&\g 0.9621& \g  0.7348&\g 0.9701  \\ \cmidrule{2-9}
&FWER&\multirow{3}{*}{0.5}  & 0.0548& 0.0000&0.0388& 0.0393&  0.0014& 0.0535\\
& a1&  &\g  0.9360&\g 0.0197&\g 0.9199&\g 0.9213& \g  0.6417&\g 0.9356\ \\
& a2&  &\g 0.9976&\g 0.2855&\g 0.9957&\g 0.9958& \g  0.9408&\g 0.9974  \\
 \bottomrule[1.5pt]
\end{tabular}
\end{center}
\end{table}

\subsection{A skewed distribution example}

Here, we consider a multivariate gamma distribution which has marginal univariate gamma distribution and a covariance structure.  To generate a multivariate gamma model, let $g_1$ be $m\times 1$ independent vectors from a gamma distribution with shape parameters $\gamma_1,$ a positive vector of dimension $m$. Define $G=K g_1$,  where $K$ is  a full rank matrix with all entries equal to either zero or one that follows some properties \citep{Ronn}.  ($K$ is called the incidence matrix).   Then $G$ has a multivariate gamma distribution with shape parameter $\alpha=K \gamma_1$ and covariance matrix $\Sigma = K \, \Gamma_1 K^T,$ where the (diagonal) matrix $\Gamma_{1}$ is the variance matrix of $g_1$. 

Given $n$ independent multivariate gamma vectors $Y=(y_1,y_2,\dots,y_n)^T,$ with $y_i=(y_{i1},\dots,y_{im})^T.$  The univariate composite log-likelihood function for the multivariate gamma model
can be formulated as
\begin{eqnarray*}
cl(\beta;Y) &=&  \sum_{i=1}^n \sum_{j=1}^m \left(
-\frac{\nu y_{ij}}{\mu_{ij}} -\nu \log \mu_{ij} + \nu log \nu + (\nu
-1) \log y_{ij} - \log \Gamma(\nu)\right),
\end{eqnarray*}
where $\mu_{ij}=E(y_{ij}),$ $\nu$ is the shape parameter, and $\mu_{ij}/\nu$ is the scale parameter. 
 We used the log link to define the mean parameter:
 $\mu_{ij} = \exp\{\vec{x}_{ij}\beta \}.$  Denote $\mu_i=(\mu_{i1},\dots,\mu_{im})^T$.
Under this set up, we have
\begin{eqnarray*}
cl^{(1)}(\beta;Y)  &=& \sum_{i=1}^{n}
\left(\frac{\partial{\mu_i}}{\partial \beta}\right)^T
V(\mu)_i^{-1}(y_i-\mu_i),
\end{eqnarray*}
where $V_i=
\mbox{diag}(\mu^2_{i1},\cdots,\mu^2_{im})/\nu$, and
\begin{eqnarray*}
H(\beta) &=& n^{-1} \sum_{i=1}^n  \left(\frac{\partial{\mu_i}}{\partial \beta}\right)^T  V_i^{-1}\left(\frac{\partial{\mu_i}}{\partial \beta}\right), \\
J(\beta) &=& n^{-1} \sum_{i=1}^n
\left(\frac{\partial{\mu_i}}{\partial \beta}\right)^T  V_i^{-1}
\, \mbox{cov}(y_i) \, V_i^{-1}
\left(\frac{\partial{\mu_i}}{\partial \beta}\right).
\end{eqnarray*}
The dispersion parameter is
$\frac{1}{\nu}=\frac{D(6(n-p)+nD)}{6(n-p)+2nD}$, where
 $D=\frac{2}{nm-p}\sum_{i,j}\left(\frac{y_{ij}-\mu_{ij}}{\mu_{ij}} + \log \frac{\mu_{ij}}{y_{ij}}\right).$
Let $\widehat V_{in}$ denote the estimator of $V_i$ obtained by substituting $\widehat \mu_{ijn}$ for $\mu_{ij}.$  We estimate $H(\beta)$ and $J(\beta)$ as
\begin{eqnarray*}
\widehat{H}_n &=& n^{-1}\sum_{i=1}^n   X_i^T  V_{in}^{-1} X_i,\\
\widehat{J}_n &=& n^{-1}\sum_{i=1}^{n} X_i^T V_{in}^{-1} \ \ \widehat{\mbox{cov}}_n(y_i) \ \
V_{in}^{-1}X_i,
\end{eqnarray*}
with empirical variance $\widehat{\mbox{cov}}_n(y_i)=
(y_i -\widehat{\mu}_{in})(y_i -
\widehat{\mu}_{in})^T$, where
where $\widehat{\mu}_i$ is the vector
$\widehat{\mu}_i=\exp\{X_i\widehat\beta_n^c\}$.

In the simulation $\nu=1$, and under the global null hypothesis $H_0,$ the true value of the regression parameters is set to $\beta=0.75,$ and the power is calculated under two different alternative configurations  $\beta_{a_1}^T=(0.75, 0.75, 0.68, 0.75, \ldots, 0.75)$ and $\beta_{a_2}^T=(0.75, 0.80, 0.68, 0.70, 0.79, 0.69, 0.75,\ldots, 0.75)$. We simulate $10\,000$ data sets with $m=3,$ and $p=10.$ We perform many-to-one comparisons with the MNQ, naive MNQ, Bonferroni, Dunn-Sid\'ak, Holm and Scheff\'e method. We consider both independent and correlated cases. We simulate with the sample size $n=3\,000$ as we found that it takes at least $n=3\,000$ for the MNQ method to have the FWER fall within 2 standard deviations away from $0.05$. This larger sample size is expected for a skewed distribution such as the multivariate gamma. Among all the methods, the MNQ method continues to achieve the highest power and exhibits the best performance. The results are presented in Table 2.

\begin{table}[h!]
\caption{FWER and power for multivariate gamma distribution} \label{tab:8}
\begin{center}
\begin{tabular}{@{}ccccccccccc@{}} \hline
\toprule[1.5pt]     &    &   MNQ & naive& Bonf & S-D  & Sch\'effe\\
\midrule[1.5pt]
FWER &independent &0.0554& 0.0507&0.0437& 0.0444&  0.0003\\
$a_1$&  &\g 0.8763&\g 0.8777&\g 0.8508&\g  0.8531&\g  0.3055 \\
$a_2$ &   &\g 0.9906&\g 0.9899&\g 0.9856& \g 0.9862&\g 0.4526\\ \hline
FWER &correlated &0.0588& 0.3427&0.0468& 0.0479&  0.0003\\
$a_1$&  &\g 0.8223&\g 0.9883&\g 0.7853& \g 0.7877& \g 0.2378 \\
$a_2$&  &\g 0.9778&\g 0.9999&\g 0.9638&\g 0.9653&\g  0.3683\\
 \bottomrule[1.5pt]
\end{tabular}
\end{center}
\end{table}

\subsection{Some technical details}

 Xu and Reid (2011) provided a detailed proof of consistency under misspecification, along with a precise list of required conditions.  One can obtain from their work sufficient conditions for consistency even in the well-specified setting.  Here, for reference, we give a proof of some asymptotic properties of the composite likelihood estimator provided that the model is correctly specified and data is formed by $n$ independent clusters, each with fixed sample size $m.$ 
%


\medskip
\noindent \textbf{Regularity conditions:}
\begin{enumerate}
\item[(A1).] The marginal density function of $y_{ij},$ $f(y;\theta)$ is distinct for different values of $y$, i.e. if $\theta_1\neq\theta_2$ then $P(f(y_{ij};\theta)\neq f(y_{ij};\theta))>0,$ for all $j=1, \ldots, m.$
\item[(A2).] The marginal densities of $y_{ij}$ have common support for all $\theta.$
\item[(A3).] The true value $\theta_0$ is an interior point of $\Omega,$ the space of possible values of the parameter $\theta.$
\item[(A4).] Let $\alpha$ and $\partial^\alpha$ denote the index and partial derivative operator, respectively, as in the standard multi-index notation from multivariable calculus.    The marginal density $\log f$ is three times continuously differentiable in a closed ball around $\theta_0.$  Moreover,  there exists a constant $c$ and an integrable function $M(y)$ such that 
\begin{eqnarray*}
\left|(\partial^{\alpha} \partial^{\theta_i}  \log f)(y;\theta) \right| \leq M(y),
\end{eqnarray*}
for all $||\theta-\theta_0||_2<c$, all $|\alpha|=2,$ and any $i=1, \ldots, p.$   Here, $||\cdot||_2$ denotes the Euclidean norm. 
\item[(A5).]  $J(\theta_0)$ is well-defined (i.e. exists and is finite) and invertible. 
\item[(A6).] $H(\theta_0)$ is well-defined (i.e. exists and is finite) and (strictly) positive-definite.   
\end{enumerate}

\bigskip
\bigskip

\noindent Define the marginal composite log-likelihood function as 
\begin{eqnarray*}
cl(\theta)&=& \log CL(\theta;Y) \ \ =  \ \ \sum_{i=1}^n  \sum_{j=1}^{m} \log f(y_{ij}; \theta),
\end{eqnarray*}
and let $cl_m(\theta; y_i)=\sum_{j=1}^m  \log f(y_{ij};\theta))$.   

\begin{thm}
Under the regularity conditions (A1)-(A6), there exists a solution to the composite likelihood equation, $\cmle,$ which satisfies
\begin{eqnarray*}
\sqrt{n}(\cmle- \theta_0) &\Rightarrow& G^{-1/2}(\theta_0) \, Z
\end{eqnarray*}
where $G(\theta) = H(\theta) J^{-1}(\theta) H(\theta),$ and $Z$ is a standard normal random vector. 
\end{thm}

\begin{proof}
The proof is divided into two main steps.   We first show that there exists a $\cmle$ which is of order $O(n^{-1/2}),$ and then we derive its asymptotic normality.    

Let $h(\theta; y) = cl(\theta; y).$   Note that for fixed $y,$ $h$ maps $\mathbb R^p$ into $\mathbb R$.  Then, by a Taylor expansion, we have that 
\begin{eqnarray*}
h(\theta; y)- h(\theta_0; y)&=&(\nabla h)(\theta_0; y)^T(\theta-\theta_0)+ (\theta-\theta_0)^T (Dh)(\theta^*; y) (\theta-\theta_0),
\end{eqnarray*}
where $\theta^*$ lies on a line joining $\theta$ and $\theta_0.$   We use $\nabla, D$ to denote the gradient and Hessian operators, respectively.  Our goal will be to show that there exists a $\theta$ in a $n^{-1/2}$ ball of $\theta_0,$ the left hand side of the above equation is negative.   This in turn will imply that there exists a CMLE which satisfies $\sqrt{n}(\cmle-\theta_0)=O_p(1).$ 

To this end, let $\theta-\theta_0=\xi M/\sqrt{n},$ with $||\xi||_2=1.$  Assume also that $||\theta-\theta_0||_2<c,$ that is, $M<c\sqrt{n}$. Then, by the above, we have
\begin{eqnarray}
&&\hspace{-2cm}\xi^T \left\{\frac{1}{\sqrt{n}}\sum_{i=1}^n (\nabla cl_m)(\theta_0,  y_i)\right\} + \xi^T \left\{\frac{1}{n}\sum_{i=1}^n (Hcl_m)(\theta^*,  y_i)\right\} \xi\notag\\\
&\equiv& \xi^T b_n M + \xi^T B_n \xi M^2,\label{line:quad}
\end{eqnarray}
where $b_n$ is a random vector converging to a mean-zero Gaussian RV, and $B_n$ is the random matrix converging to the negative definite matrix $-H(\theta_0)$.  The first of these follows by the central limit theorem, along with assumption (A5).   The second follows by applying the law of large numbers, along with assumptions (A4) and (A6).    Note that the second fact implies also that the eigenvalues of $B_n$ converge almost surely to the eigenvalues of $-H(\theta_0).$  

Let $\lambda^{(p)}_n$ denote the largest eigenvalue of $-B_n,$ and let $S=\{\xi:  ||\xi||_2=1\}.$    Since $b_n$ converges as a random Gaussian vector (with mean zero), and $\xi^Tb_n$ is uniformly continuous on $S,$ it follows that $\xi^Tb_n$ converges to a mean-zero Gaussian process in $C(S),$ the space of continuous functions on $S$ endowed with the uniform metric.   This implies that $\xi^Tb_n$ is tight in $C(S)$, and hence for all $\epsilon>0,$ there exists an $M_\epsilon,$ such that
\begin{eqnarray*}
\limsup_n P\left(\sup_{\xi \in S} \xi^T b_n/\lambda_n^{(p)} < M_\epsilon \right) &\geq& 1-\epsilon.
\end{eqnarray*}
Then, by \eqref{line:quad}, if $\xi^T b_n/\lambda_n^{(p)} < M,$ then $\xi^T b_n M + \xi^T B_n \xi M^2 <0,$  which in turn implies that 
\begin{eqnarray*}
\limsup_n P\left(\xi^T b_n M_\epsilon + \xi^T B_n \xi M_\epsilon^2 <0 \ \ \forall  \xi \in S\right) &\geq& 1-\epsilon.
\end{eqnarray*}

Note that if $\xi^T b_n M_\epsilon + \xi^T B_n \xi M_\epsilon^2 <0 \ \ \forall  \xi \in S,$
then, by the above and continuity of $cl_m,$ this implies that for sufficiently large $n,$ (with a probability of at least $1-\epsilon$) there exists at least one local maximum on the set $B_{M_\epsilon/\sqrt{n}}(\theta_0)\cap B_{c}(\theta_0).$  This implies that there exists a $\cmle$ which satisfies $\sqrt{n}(\cmle-\theta_0)=O_p(1).$

Let $g(\theta; y)=cl^{(1)}_m(\theta; y) =\nabla cl_m(\theta; y)$ (this is the vector of first derivatives), then using a multivariate Taylor expansion, we have that 
\begin{eqnarray*}
g(\cmle; y)&=&g(\theta_0; y)+ \sum_{|\alpha|\leq 1} (\partial^\alpha g)(\theta_0; y)(\cmle- \theta_0)^\alpha\\
&&  \ \ + \sum_{|\alpha|=2} \frac{2}{\alpha!}(\cmle- \theta_0)^\alpha \int_0^1 (1-t) (\partial^\alpha g)(\theta_0+t(\cmle-\theta_0); y)dt,
\end{eqnarray*}
again using the multi-index notation.  We take $\cmle$ to be the local maximizer found above.  
This time, for fixed $ y,$ $g$ maps $\mathbb R^p$ into $\mathbb R^p$, so we have chosen to bound the error term a little differently than above.   
We let $R_{n,i}$ denote the third term on the right hand side of this equation when $y$ is replaced with $y_i$.   Next, as by definition $\sum_{i=1}^n cl^{(1)}_m(\cmle; y_i)=0,$ we have that 
\begin{eqnarray}\label{line:taylor}
\frac{1}{\sqrt{n}}\sum_{i=1}^n(Dcl_m)(\theta; y_i)^T (\cmle- \theta_0) +\frac{1}{\sqrt{n}}\sum_{i=1}^nR_{n,i} &=& \frac{1}{\sqrt{n}}\sum_{i=1}^n f(\theta_0; y_i). \hspace{1cm}
\end{eqnarray}

By condition (A4), we have that 
\begin{eqnarray*}
&&\hspace{-2cm}\left|\sum_{|\alpha|=2} \frac{2}{\alpha!}(\cmle- \theta_0)^\alpha \int_0^1 (1-t) (D^\alpha g)(\theta_0+t(\cmle-\theta_0); y)dt\right|\\
&\leq& \sum_{|\alpha|=2}  \frac{1}{\alpha!}  |\cmle- \theta_0|^\alpha |M(y)|,
\end{eqnarray*}
from which it follows that, 
\begin{eqnarray*}
\left|\frac{1}{\sqrt{n}}\sum_{i=1}^nR_{n,i}\right| &\leq & \left\{\sqrt{n}||\cmle-\theta_0||_2^2\right\} \left\{ \frac{1}{n}\sum_{i=1}^n |M(y_i)|\right\}.
\end{eqnarray*}
The first term is then $o_p(1)$ by the first part of this proof, and by the law of large numbers (since $M$ is integrable), the second term is $O_p(1).$     Next, consider
\begin{eqnarray*}
\sqrt{n} \left\{\frac{1}{n}\sum_{i=1}^ncl^{(2)}_m(\theta; y_i)-H(\theta_0)\right\}(\cmle- \theta_0).
\end{eqnarray*}
By similar argument to that above, this is also $o_p(1).$  This allows us to re-write \eqref{line:taylor} as
\begin{eqnarray*}
\sqrt{n}H(\theta_0)(\cmle- \theta_0)  &=& \frac{1}{\sqrt{n}}\sum_{i=1}^n f(\theta_0; y_i) + o_p(1)
\end{eqnarray*}
A straightforward application of the central limit theorem shows that the term on the right hand side has a Gaussian limiting distribution with mean zero and variance $J(\theta_0).$   The full result follows.
\end{proof}

\bibliographystyle{ims}
\bibliography{multiple_arxive_REV}

\begin{thebibliography}{27}
\expandafter\ifx\csname natexlab\endcsname\relax\def\natexlab#1{#1}\fi
\expandafter\ifx\csname url\endcsname\relax
  \def\url#1{\texttt{#1}}\fi
\expandafter\ifx\csname urlprefix\endcsname\relax\def\urlprefix{URL }\fi

\bibitem[{Barua(2011)}]{baru}
\textsc{Barua, G. M.~N. B.~M., A.} (2011).
\newblock Prevalence of depressive disorders in the elderly \textbf{31}
  620--624.

\bibitem[{Benjamini and Hochberg(1995)}]{benjHoch}
\textsc{Benjamini, Y.} and \textsc{Hochberg, Y.} (1995).
\newblock Controlling the false discovery rate: a practical and powerful
  approach to multiple testing.
\newblock \textit{J. Roy. Statist. Soc. Ser. B} \textbf{57} 289--300.

\bibitem[{Besag(1974)}]{besa}
\textsc{Besag, J.} (1974).
\newblock Spatial interaction and the statistical analysis of lattice systems.
\newblock \textit{J. Roy. Statist. Soc. Ser. B} \textbf{36} 192--236.
\newblock With discussion by D. R. Cox, A. G. Hawkes, P. Clifford, P. Whittle,
  K. Ord, R. Mead, J. M. Hammersley, and M. S. Bartlett and with a reply by the
  author.

\bibitem[{Bretz et~al.(2010)Bretz, Hothorn and Westfall}]{bret}
\textsc{Bretz, F.}, \textsc{Hothorn, T.} and \textsc{Westfall, P.} (2010).
\newblock \textit{Multiple Comparisons Using R}.
\newblock Chapman and Hall/CRC Press, Boca Raton, Florida, USA.

\bibitem[{Cox and Reid(2004)}]{coxReid}
\textsc{Cox, D.~R.} and \textsc{Reid, N.} (2004).
\newblock A note on pseudolikelihood constructed from marginal densities.
\newblock \textit{Biometrika} \textbf{91} 729--737.

\bibitem[{Geys et~al.(1997)Geys, Molenberghs and Ryan}]{geysMole}
\textsc{Geys, H.}, \textsc{Molenberghs, G.} and \textsc{Ryan, L.~M.} (1997).
\newblock Pseudo-likelihood inference for clustered binary data.
\newblock \textit{Comm. Statist. Theory Methods} \textbf{26} 2743--2767.

\bibitem[{Hochberg and Tamhane(1987)}]{hoch}
\textsc{Hochberg, Y.} and \textsc{Tamhane, A.} (1987).
\newblock \textit{Multiple Comparison Procedures}.
\newblock New York: Willy.

\bibitem[{Holm(1979)}]{holm}
\textsc{Holm, S.} (1979).
\newblock A simple sequentially rejective multiple test procedure.
\newblock \textit{Scand. J. Statist.} \textbf{6} 65--70.

\bibitem[{Hommel(1988)}]{homm}
\textsc{Hommel, G.} (1988).
\newblock A stagewise rejective multiple test procedure based on a modified
  bonferroni test  383--386.

\bibitem[{Hothorn et~al.(2008{\natexlab{a}})Hothorn, Bretz and
  Westfall}]{hothBretWest}
\textsc{Hothorn, T.}, \textsc{Bretz, F.} and \textsc{Westfall, P.}
  (2008{\natexlab{a}}).
\newblock Simultaneous inference in general parametric models.
\newblock \textit{Biom. J.} \textbf{50} 346--363.

\bibitem[{Hothorn et~al.(2008{\natexlab{b}})Hothorn, Bretz, Westfall and
  Heiberger}]{hoth}
\textsc{Hothorn, T.}, \textsc{Bretz, F.}, \textsc{Westfall, P.} and
  \textsc{Heiberger, R.~M.} (2008{\natexlab{b}}).
\newblock multcomp: Simultaneous inference in general parametric models .

\bibitem[{Konietschke et~al.(2013)Konietschke, Bosiger, Brunner and
  Hothorn}]{koniBosiBrun}
\textsc{Konietschke, F.}, \textsc{Bosiger, S.}, \textsc{Brunner, E.} and
  \textsc{Hothorn, L.~A.} (2013).
\newblock Are multiple contrast tests superior to the anova?
\newblock \textit{Int. J. Biostat.} \textbf{9} 11.

\bibitem[{Konietschke et~al.(2012)Konietschke, Hothorn and
  Brunner}]{koniHothBrun}
\textsc{Konietschke, F.}, \textsc{Hothorn, L.~A.} and \textsc{Brunner, E.}
  (2012).
\newblock Rank-based multiple test procedures and simultaneous confidence
  intervals.
\newblock \textit{Electron. J. Stat.} \textbf{6} 738--759.

\bibitem[{Lindsay(1988)}]{lind}
\textsc{Lindsay, B.~G.} (1988).
\newblock Composite likelihood methods.
\newblock In \textit{Statistical inference from stochastic processes ({I}thaca,
  {NY}, 1987)}, vol.~80 of \textit{Contemp. Math.} Amer. Math. Soc.,
  Providence, RI, 221--239.

\bibitem[{Lisovskaja(2015)}]{lisoBurm}
\textsc{Lisovskaja, B. C.~F., V.} (2015).
\newblock A decision theoretic approach to optimization of multiple testing
  procedures \textbf{57} 64--75.

\bibitem[{Meijer(2015)}]{meijGoem}
\textsc{Meijer, G.~J.~J., R.~J.} (2015).
\newblock A multiple testing method for hypotheses structured in a directed
  acyclic graph \textbf{57} 123--143.

\bibitem[{Molenberghs and Ryan(1999)}]{moleRyan}
\textsc{Molenberghs, G.} and \textsc{Ryan, L.~M.} (1999).
\newblock An exponential family model for clustered multivariate binary data
  \textbf{10} 279--300.

\bibitem[{Renard et~al.(2004)Renard, Molenberghs and Geys}]{renaMole}
\textsc{Renard, D.}, \textsc{Molenberghs, G.} and \textsc{Geys, H.} (2004).
\newblock A pairwise likelihood approach to estimation in multilevel probit
  models.
\newblock \textit{Comput. Statist. Data Anal.} \textbf{44} 649--667.

\bibitem[{Ronning(1997)}]{Ronn}
\textsc{Ronning, G.} (1997).
\newblock A simple scheme for generating multivariate gamma distributions with
  non-negative covariance matrix \textbf{19} 179--183.

\bibitem[{Sch\'effe(1959)}]{sche}
\textsc{Sch\'effe} (1959).
\newblock \textit{The analysis of variance}.
\newblock Wiley, New York.

\bibitem[{Sidak(1968)}]{sida}
\textsc{Sidak, Z.} (1968).
\newblock On multivariate normal probabilities of rectangles: Their dependence
  on correlations.
\newblock \textit{Ann. Math. Statist.} \textbf{39} 1425--1434.

\bibitem[{Simes(1986)}]{sime}
\textsc{Simes, R.~J.} (1986).
\newblock An improved {B}onferroni procedure for multiple tests of
  significance.
\newblock \textit{Biometrika} \textbf{73} 751--754.

\bibitem[{Varin(2008)}]{vari}
\textsc{Varin, C.} (2008).
\newblock On composite marginal likelihoods.
\newblock \textit{AStA Adv. Stat. Anal.} \textbf{92} 1--28.

\bibitem[{Varin et~al.(2011)Varin, Reid and Firth}]{variReid}
\textsc{Varin, C.}, \textsc{Reid, N.} and \textsc{Firth, D.} (2011).
\newblock An overview of composite likelihood methods.
\newblock \textit{Statist. Sinica} \textbf{21} 5--42.

\bibitem[{Varin and Vidoni(2005)}]{variVido}
\textsc{Varin, C.} and \textsc{Vidoni, P.} (2005).
\newblock A note on composite likelihood inference and model selection.
\newblock \textit{Biometrika} \textbf{92} 519--528.

\bibitem[{Xu and Reid(2011)}]{xuReid}
\textsc{Xu, X.} and \textsc{Reid, N.} (2011).
\newblock On the robustness of maximum composite likelihood estimate.
\newblock \textit{J. Statist. Plann. Inference} \textbf{141} 3047--3054.

\bibitem[{Zhao and Joe(2005)}]{zhaoJoe}
\textsc{Zhao, Y.} and \textsc{Joe, H.} (2005).
\newblock Composite likelihood estimation in multivariate data analysis.
\newblock \textit{Canad. J. Statist.} \textbf{33} 335--356.

\end{thebibliography}

\end{document}